\documentclass[11pt]{article}

\usepackage{amsmath,amsfonts,amssymb,amsthm}
\usepackage{hyperref}
\renewcommand{\geq}{\geqslant}
\renewcommand{\leq}{\leqslant}
\renewcommand{\preceq}{\preccurlyeq}
\renewcommand{\succeq}{\succcurlyeq}

\newcommand{\SC}{Silvio Capobianco}
\newcommand{\TTU}{Tallinn University of Technology}
\newcommand{\DSS}{Department of Software Science}
\newcommand{\ADDR}{Akadeemia tee 21/1, 12618 Tallinn, Estonia}

\newcommand{\ATTTU}{\href{mailto:silvio.capobianco@taltech.ee}{\texttt{silvio.capobianco@taltech.ee}}}
\newcommand{\grazzie}{%
  This research was supported by the Estonian Ministry of Education
  and Research institutional research grant no. IUT33-13 and by the
  Estonian Research Council grant PRG1210.
}

\title{Fekete's lemma for componentwise subadditive functions of two or more real variables}
\author{%
  \SC \footnote{\DSS, \TTU. \ADDR. \ATTTU} \footnote{\grazzie}%
}
\date{}

\newcommand{\Acal}{\ensuremath{\mathcal{A}}}
\newcommand{\deux}{\ensuremath{\{0,1\}}}

\newcommand{\Fcal}{\ensuremath{\mathcal{F}}}

\newcommand{\ifthenelse}[4]{%
  \ensuremath{%
    \left\{ \begin{array}{ll}
      {#2} & \mathrm{if} \; {#1} \,, \\
      {#3} & \mathrm{otherwise} \,{#4}
    \end{array} \right.
  }
}
\newcommand{\integers}{\ensuremath{\mathbb{Z}}}

\newcommand{\lambdaxt}[2]{
  \ensuremath{{#1} \mapsto {#2}}
}

\newcommand{\negativeint}{\ensuremath{\integers_{-}}}
\newcommand{\negatives}{\ensuremath{\reals_{-}}}
\newcommand{\mainorthant}{\ensuremath{\orthant_{+}}}

\newcommand{\orthant}{\ensuremath{\mathcal{R}}}

\newcommand{\pf}{\ensuremath{\Pcal\Fcal}}
\newcommand{\Pcal}{\ensuremath{\mathcal{P}}}
\newcommand{\positiveint}{\ensuremath{\integers_{+}}}
\newcommand{\positives}{\ensuremath{\reals_{+}}}

\newcommand{\reals}{\ensuremath{\mathbb{R}}}

\newcommand{\slice}[2]{\ensuremath{\left[{#1}\!:\!{#2}\right]}}
\newcommand{\Ucal}{\ensuremath{\mathcal{U}}}
\newcommand{\Vcal}{\ensuremath{\mathcal{V}}}

\newtheorem{theorem}{Theorem}[section]
\newtheorem{corollary}[theorem]{Corollary}
\newtheorem{lemma}[theorem]{Lemma}
\newtheorem{proposition}[theorem]{Proposition}

\theoremstyle{definition}
\newtheorem{definition}[theorem]{Definition}
\newtheorem{example}[theorem]{Example}

\begin{document}

\maketitle

\begin{abstract}
  We prove an analogue of Fekete's subadditivity lemma for functions
  of several real variables which are subadditive in each variable
  taken singularly. This extends both the classical case for
  subadditive functions of one real variable, and a result in a
  previous paper by the author. While doing so, we prove that the
  functions with the property mentioned above are bounded in every
  closed and bounded subset of their domain. The arguments expand on
  those in Chapter 6 of E. Hille's 1948 textbook.
\end{abstract}

\section{Introduction}

A real-valued function $f$ defined on a semigroup $(S,\cdot)$ is
\emph{subadditive} if
\begin{equation}
  \label{eq:subadd}
  f(x \cdot y) \leq f(x) + f(y)
\end{equation}
for every $x,y\in{S}$. Examples of subadditive functions include the
absolute value of a complex number; the ceiling of a real number
(smallest integer not smaller than it); the cardinality of a a finite
subset of a given set; and the length of a word over an
alphabet. Subadditive functions have applications in many fields
including information theory \cite{lm1995}, economics, and
combinatorics. In real analysis, subadditive functions are much more
interesting than additive ones: for example, every additive, Lebesgue
measurable function of one real variable is linear
(cf. \cite{sierp1920}) but the characteristic function of the set of
irrational numbers is subadditive, Lebesgue measurable, not linear,
and everywhere discontinuous.

A classical result in mathematical analysis, \emph{Fekete's lemma}
\cite{fekete1923} states that, if $f$ is a real-valued subadditive
function of one positive integer or positivereal variable, then
$f(x)/x$ converges, for $x\to+\infty$, to its greatest lower
bound. This simple fact has a huge number of applications in many
fields, including symbolic dynamics (cf. \cite[Chapter 4]{lm1995}) and
the theory of neural networks (see \cite{guerratoninelli}). Reusing a
metaphor from \cite{capobianco2008}, Fekete's lemma says that for a
sequence of independent observations, the \emph{average information
  per observation} converges to its greatest lower bound.

Given the importance and ubiquity of Fekete's lemma, we wonder if
similar results may hold for functions of many variables. Oddly, the
mathematical literature seems to contain generalizations where, in
almost all cases, the function in the limit actually depends again on
a single variable, which is sometimes a real number, sometimes a
finite set; many of these are closer to a corollary than to an
extension. Of course there are practical reasons for this: for one,
division by a vector with many components is undefined. But maybe we
could look for a different \emph{type} of limit, or even for a
different \emph{flavor} of subadditivity.

At the aim of understanding which of the above could be feasible, we
note that, if
\begin{math}
  S = S_1 \times S_2
\end{math}
is a product semigroup, we can also consider the case of a function
which is subadditive in \emph{each variable, however given the other}.
That is, instead of requiring
\begin{math}
  f(x_1 y_1, x_2 y_2) \leq f(x_1, x_2) + f(y_1, y_2)
\end{math}
for every $x_1$, $x_2$, $y_1$, and $y_2$, we could demand that:
\begin{enumerate}
\item
  \label{it:subadd-componentwise-x1}
  \begin{math}
    f(x_1 y_1, x_2) \leq f(x_1, x_2) + f(y_1, x_2)
  \end{math}
  for every $x_1$, $x_2$, and $y_1$; and
\item
  \label{it:subadd-componentwise-x2}
  \begin{math}
    f(x_1, x_2 y_2) \leq f(x_1, x_2) + f(x_1, y_2)
  \end{math}
  for every $x_1$, $x_2$, and $y_2$.
\end{enumerate}
The two requirements above, even together, do not imply subadditivity
as a function defined on the product semigroup, nor does the latter
imply the former: see Example
\ref{ex:subadd-separate-not-joint}. Oddly again, this multivariate
``componentwise subadditivity'' seems not to have been addressed very
often in the literature.

In this paper, we state and prove an extension of Fekete's lemma to
componentwise subadditive functions of $d\geq2$ real variables. We
state a special case as an example, leaving the full statement to
Section \ref{sec:fekete-d}.
\begin{proposition}
  \label{prop:fekete-2d-positives}
  Let $f$ be a function of two positive real variables which is
  subadditive in each of them, however given the other. For every
  $\delta>0$ there exists $R>0$ such that, if both $x_1>R$ and
  $x_2>R$, then:
  \begin{displaymath}
    \frac{f(x_1, x_2)}{x_1 \cdot x_2}
    <
    \inf_{x_1, x_2 > 0} \frac{f(x_1, x_2)}{x_1 \cdot x_2}
    + \delta \,.
  \end{displaymath}
  In addition,
  \begin{displaymath}
    \lim_{x_1 \to +\infty} \lim_{x_2 \to +\infty} \frac{f(x_1, x_2)}{x_1 \cdot x_2}
    =
    \lim_{x_2 \to +\infty} \lim_{x_1 \to +\infty} \frac{f(x_1, x_2)}{x_1 \cdot x_2}
    =
    \inf_{x_1, x_2 > 0} \frac{f(x_1, x_2)}{x_1 \cdot x_2} \,.
  \end{displaymath}
\end{proposition}
That is: if the componentwise subadditive function $f(x,y)$ is
considered as a \emph{net} on the \emph{directed set} of pairs of
positive reals with the \emph{product ordering} where
$(x_1,x_2)\leq(y_1,y_2)$ if and only if $x_1\leq{y_1}$ and
$x_2\leq{y_2}$, then the \emph{simultaneous limit} on this directed
set of the net $\dfrac{f(x_1,x_2)}{x_1\cdot{x_2}}$ is its greatest
lower bound. This is a generalization of the original statement, where
the functions depend on multiple \emph{independent} real variables,
both notions of subadditivity and limit are extended, and the original
lemma is a special case for $d=1$. The double limit is also
remarkable, because multiple limits need not commute, let alone
coincide with a simultaneous limit.

A similar statement for functions defined on $d$-tuples of positive
integers (instead of reals) was proved in \cite{capobianco2008}; see
also \cite{lmp2013} for an application. The argument presented there,
however, relies on a hidden hypothesis of boundedness on compact
subsets, which comes for free in the integer setting (where compact
subsets are precisely the finite subsets) but must be proved in the
new one, and \emph{cannot} be inferred from boundedness in each
variable however given the others (see Example
\ref{ex:bound-unbound}). By adapting the proof of \cite[Theorem
  6.4.1]{hille1948} we obtain the following result: componentwise
subadditive functions defined on suitable regions of $\reals^d$ are
indeed bounded on compact subsets. For $d=2$ and positive variables
the statement goes as follows:
\begin{proposition}
  \label{prop:subadd-bounded-1q}
  In the hypotheses of Proposition \ref{prop:fekete-2d-positives}, the
  function $f$ is bounded on $[a,b]\times[c,d]$ for every $0<a<b$ and
  $0<c<d$.
\end{proposition}

The paper is organized as follows. Section \ref{sec:background}
provides the theoretical background. In Section
\ref{sec:subadd-componentwise} we introduce componentwise
subadditivity and explain how it is different from subadditivity in
the product semigroup. In Section \ref{sec:boundedness} we adapt the
argument from \cite[Theorem 6.4.1]{hille1948} to prove that
componentwise subadditive functions of $d$ real variables are bounded
on compact subsets of $\reals^d$. In Section \ref{sec:fekete-d} we
state, prove, and discuss the main theorem: boundedness will have a
crucial role in the proof. Section \ref{sec:owlemma} is a discussion
on how the beautiful \emph{Ornstein-Weiss lemma}
\cite{ornstein-weiss}, an important result on subadditive functions
defined on finite subsets of groups of a certain class which includes
$\integers$ and $\reals$, \emph{is not} an extension of Fekete's
lemma.

\section{Background}
\label{sec:background}

Throughout the paper, the subsets of $\reals^d$ and the real-valued
functions of real variables are presumed to be Lebesgue measurable. We
also let real-valued functions take value either $+\infty$ or
$-\infty$, but not both.

We denote by $\reals$, $\positives$, and $\negatives$ the sets of real
numbers, positive real numbers, and negative real numbers,
respectively. Similarly, we denote by $\integers$, $\positiveint$, and
$\negativeint$ the sets of integers, positive integers, and negative
integers, respectively. All these sets are considered as additive
semigroups (groups in the case of $\reals$ and $\integers$). If $m$
and $n$ are integers and $m\leq{n}$ we denote the \emph{slice}
\begin{math}
  \{ m, m+1, \ldots, n-1, n \} = \left[ m, n \right] \cap \integers
\end{math}
as $\slice{m}{n}$. If $X$ is a set, we denote by $\pf(X)$ the set of
its finite subsets.

If the sets $X$ and $Y$ where the variable $x$ and the expression
$E(x)$ take values are irrelevant or clear from the context, we denote
by
\begin{math}
  \lambdaxt{x}{E(x)}
\end{math}
the function that associates to each value $\overline{x}$ taken by $x$
the value $E(\overline{x})$. For example,
\begin{math}
  \lambdaxt{x}{1}
\end{math}
is the constant function taking value 1 everywhere.

A \emph{directed set} is a partially ordered set
\begin{math}
  \Ucal = (U, \preceq)
\end{math}
with the following additional property: for every $x,y\in{U}$ there
exists $z\in{U}$ such that $x\preceq{z}$ and $y\preceq{z}$.  Every
totally ordered set is a directed set, and so is the family of
\emph{decompositions}
\begin{math}
  x = \left\{ a = x_0 < x_1 < \ldots < x_n = b \right\}
\end{math}
of the compact interval $[a,b]$ with the partial order $x\preceq{y}$
iff for every $i$ there exists $j$ such that $x_i=y_j$.  A function
$f$ defined on $U$ is also called a \emph{net} on $\Ucal$. If $Y$ is
the codomain of $f$, a \emph{subnet} of $f$ is a net
\begin{math}
  g : V \to Y
\end{math}
on a directed set
\begin{math}
  \Vcal = (V, \leq)
\end{math}
together with a function
\begin{math}
  \phi : V \to U
\end{math}
such that:
\begin{enumerate}
\item
  \begin{math}
    f \circ \phi = g
  \end{math};
  and
\item
  for every $x\in{U}$ there exists $y\in{V}$ such that, if $z\in{V}$
  and $z\geq{y}$, then $\phi(z)\succeq{x}$.
\end{enumerate}
For example, a subsequence
\begin{math}
  \{ x_{n_k} \}_{k \geq 1}
\end{math}
of a sequence
\begin{math}
  \{ x_{n} \}_{n \geq 1}
\end{math}
of real numbers is a subnet, with
\begin{math}
  V = U = \positiveint
\end{math}
and
\begin{math}
  \phi(k) = n_k
\end{math}.

If
\begin{math}
  \Ucal = (U, \preceq)
\end{math}
is a directed set and $f:U\to\reals$ is a function, the \emph{lower
  limit} and the \emph{upper limit} of $f$ in $\Ucal$ are the values
\begin{equation}
  \label{eq:liminf-net}
  \liminf_{x \to \Ucal} f(x)
  =
  \sup_{x \in U} \inf_{y \succeq x} f(y)
\end{equation}
and
\begin{equation}
  \label{eq:limsup-net}
  \limsup_{x \to \Ucal} f(x)
  =
  \inf_{x \in U} \sup_{y \succeq x} f(y)
  \,,
\end{equation}
respectively. The following chain of equalities holds:
\begin{equation}
  \label{eq:liminf-limsup}
  \inf_{x \in U} f(x)
  \leq
  \liminf_{x \to \Ucal} f(x)
  \leq
  \limsup_{x \to \Ucal} f(x)
  \leq
  \sup_{x \in U} f(x)
  \,.
\end{equation}
Moreover, if
\begin{math}
  \Vcal = (V, \leq)
\end{math}
is a directed set and
\begin{math}
  g : V \to \reals
\end{math}
is a subnet of $f$, then:
\begin{equation}
  \label{eq:liminf-limsup-subnet}
  \liminf_{x \to \Ucal} f(x)
  \leq
  \liminf_{y \to \Vcal} g(y)
  \leq
  \limsup_{y \to \Vcal} g(y)
  \leq
  \limsup_{x \to \Ucal} f(x)
  \,.
\end{equation}
If
\begin{math}
  \liminf_{x \to \Ucal} f(x) = \limsup_{x \to \Ucal} f(x)
\end{math},
their common value $L$ is called the \emph{limit} of $f$ in $\Ucal$,
and we say that $f$ \emph{converges} to $L$ in $\Ucal$. This is
equivalent to the following: for every $\varepsilon>0$ there exists
\begin{math}
  x_\varepsilon \in U
\end{math}
such that
\begin{math}
  |f(x) - L| < \varepsilon
\end{math}
for every
\begin{math}
  x \succeq x_\varepsilon
\end{math}.
In this case, every subnet of $f$ also converges to $L$.

The \emph{ordered product} of a family
\begin{math}
  \{ (X_i, \leq_i) \}_{i \in I}
\end{math}
of ordered sets is the ordered set $(X,\leq_\Pi)$ where
\begin{math}
  X = \prod_{i \in I} X_i
\end{math}
and the \emph{product ordering} $\leq_\Pi$ is defined as:
\begin{equation}
  \label{eq:product-order}
  x \leq_\Pi y
  \Longleftrightarrow
  x_i \leq y_i \;\; \textrm{for every } i \in I \,.
\end{equation}
If each $(X_i,\leq_i)$ is a directed set, then so is
$(X,\leq_\Pi)$. For $d\geq2$ and $w\in\deux^d$ the \emph{orthant}
denoted by $w$ is the directed set
\begin{math}
  \orthant_w = \prod_{i=1}^d (X_i, \leq_i)
\end{math}
where
\begin{math}
  (X_i, \leq_i) = (\positives, \leq)
\end{math}
if $w_i=0$ and
\begin{math}
  (X_i, \leq_i) = (\negatives, \geq)
\end{math}
if $w_i=1$. For example, $\orthant_{10}$ is the open second quadrant
of the Cartesian plane, with
\begin{math}
  (x_1, x_2) \leq (y_1, y_2)
\end{math}
if and only if $x_1\geq{y_1}$ and $x_2\leq{y_2}$. In particular, the
\emph{main orthant} of $\reals^d$, corresponding to $w=0^d$, is
\begin{math}
  \mainorthant^d = \left( \positives^d, \leq_\Pi \right)
\end{math}.
Note that, if
\begin{math}
  f : \positives^d \to \reals
\end{math}
is a net on $\mainorthant^d$ and
\begin{math}
  \{ x_{i, n} \}_{n \geq 1}
\end{math},
$i\in\slice{1}{d}$, are sequences of positive reals such that
\begin{math}
  \lim_{n \to \infty} x_{i,n} = +\infty
\end{math}
for every $i\in\slice{1}{d}$, then
\begin{math}
  g(n) = f \left( x_{1,n}, \ldots, x_{d,n} \right)
\end{math}
is a subnet of $f$: consequently, if $f$ converges to $L\in\reals$ in
$\mainorthant^d$, then $g(n)$ converges to $L$ for $n\to\infty$.

\section{Componentwise subadditivity}
\label{sec:subadd-componentwise}

In the literature, subadditivity is most often studied in functions of
a single variable, which sometimes may be vector rather than
scalar. But in some cases, it is of interest to consider functions of
$d$ independent variables, which are subadditive when considered as
functions of only one of those, but however given the remaining ones.

\begin{definition}
  \label{def:subadd-componentwise}
  Let
  \begin{math}
    S_1, \ldots, S_d
  \end{math}
  be semigroups, let
  \begin{math}
    S = \prod_{i=1}^d S_i
  \end{math},
  and let
  \begin{math}
    f : S \to \reals
  \end{math}.
  Given $i\in\slice{1}{d}$, we say that $f$ is \emph{subadditive in
    $x_i$ independently of the other variables} if, however given
  \begin{math}
    x_j \in S_j
  \end{math}
  for every
  \begin{math}
    j \in \slice{1}{d} \setminus \{ i \}
  \end{math},
  the function
  \begin{math}
    \lambdaxt{x_i}{f(x_1, \ldots, x_i, \ldots, x_d)}
  \end{math}
  is subadditive on $S_i$. We say that $f$ is \emph{componentwise
    subadditive} if it is subadditive in each variable independently
  of the others.
\end{definition}

\begin{example}
  \label{ex:subadd-componentwise-immediate}
  If
  \begin{math}
    f_1 : S_1 \to \reals
  \end{math}
  and
  \begin{math}
    f_2 : S_2 \to \reals
  \end{math}
  are both subadditive and nonnegative, then
  \begin{math}
    f : S_1 \times S_2 \to \reals
  \end{math}
  defined by
  \begin{math}
    f(x_1, x_2) = f_1(x_1) \cdot f_2(x_2)
  \end{math}
  is componentwise subadditive.

  If one between $f_1$ and $f_2$ takes negative values, then $f$ might
  not be componentwise subadditive. For example, $f(x_1)=-x_1$ is
  subadditive on $\positives$, because it is linear, and
  $f_2(x_2)=\sqrt{x_2}$ is also subadditive on $\positives$, because
  it is nondecreasing and
  \begin{math}
    x_2 + y_2 < \left( \sqrt{x_2} + \sqrt{y_2} \right)^2
  \end{math}
  for every $x_2,y_2>0$; but for any fixed $x_1>0$, the function
  \begin{math}
    \lambdaxt{x_2}{-x_1 \sqrt{x_2}}
  \end{math}
  is not subadditive on $\positives$.
\end{example}

\begin{example} (cf. \cite[Section 3]{capobianco2008})
  \label{ex:subadd-subshift-logoutput}
  Let $d$ be a positive integer and let $A$ be a finite set with
  $a\geq2$ elements, considered as a discrete space. The
  \emph{translation} by $v\in\integers^d$ is the function
  \begin{math}
    \sigma_v : A^{\integers^d} \to A^{\integers^d}
  \end{math}
  defined by
  \begin{math}
    \sigma_v(c)(x) = c(x+v)
  \end{math}
  for every $x\in\integers^d$. A $d$-dimensional \emph{subshift} on
  $A$ is a subset $X$ of $A^{\integers^d}$ which is closed in the
  product topology and invariant by translation, that is, if
  $c\in{X}$, then $\sigma_v(c)\in{X}$ for every
  $v\in\integers^d$. Given $d$ positive integers $n_1,\ldots,n_d$, an
  \emph{allowed pattern} of sides $n_1,\ldots,n_d$ for $X$ is a
  function
  \begin{math}
    p : \prod_{i=1}^d \slice{1}{n_i} \to A
  \end{math}
  such that there exists $c\in{X}$ for which the restriction of $c$ to
  $\prod_{i=1}^d\slice{1}{n_i}$ coincides with $p$. Let
  $\Acal_X(n_1,\ldots,n_d)$ be the number of allowed patterns for $X$
  of sides $n_1,\ldots,n_d$: then
  \begin{equation}
    \label{eq:subadd-ca-logoutput}
    f(n_1, \ldots, n_d)
    =
    \log_{a} \Acal_X(n_1, \ldots, n_d)
    \;\;
    \textrm{for every} \; n_1, \ldots, n_d \in \positiveint
  \end{equation}
  is componentwise subadditive, because every allowed pattern of sides
  \begin{math}
    n_1 + m_1, n_2, \ldots, n_d
  \end{math}
  can be obtained by joining an allowed pattern of sides $n_1,n_2,$
  $\ldots,n_d$ with an allowed pattern of sides $m_1,n_2,\ldots,n_d$,
  but joining two such allowed patterns does not necessarily produce
  an allowed pattern; similarly for the other $d-1$ coordinates. This
  works because $X$ is invariant by translations.
\end{example}

Componentwise subadditivity is very different from subadditivity with
respect to the operation of the product semigroup. Already with $d=2$,
if
\begin{math}
  f : S_1 \times S_2 \to \reals
\end{math}
is subadditive, then for every
\begin{math}
  x_1, y_1 \in S_1
\end{math}
and
\begin{math}
  x_2, y_2 \in S_2
\end{math}
we have:
\begin{equation}
  \label{eq:subadditive-2}
  f(x_1 y_1, x_2 y_2)
  \leq
  f(x_1, x_2) + f(y_1, y_2)
  \,,
\end{equation}
while if $f$ is componentwise subadditive, then for every
\begin{math}
  x_1, y_1 \in S_1
\end{math}
and
\begin{math}
  x_2, y_2 \in S_2
\end{math}
we have the more complex upper bound:
\begin{equation}
  \label{eq:subadditive-2-componentwise}
  f(x_1 y_1, x_2 y_2)
  \leq
  f(x_1, x_2) + f(x_1, y_2) + f(y_1, x_2) + f(y_1, y_2) \,.
\end{equation}
If $f$ is nonnegative, then (\ref{eq:subadditive-2}) implies
(\ref{eq:subadditive-2-componentwise}), which however is weaker than
the conditions of Definition \ref{def:subadd-componentwise}; if $f$ is
nonpositive, then (\ref{eq:subadditive-2-componentwise}) implies
(\ref{eq:subadditive-2}). In general, however, neither implies the
other.

\begin{example}
  \label{ex:subadd-separate-not-joint}
  By our discussion in Example
  \ref{ex:subadd-componentwise-immediate}, the function
  \begin{math}
    f(x_1, x_2) = \sqrt{x_1 x_2}
  \end{math}
  is componentwise subadditive on $\positives^2$. However, $f$ is not
  subadditive, because
  \begin{math}
    f(3, 3) = 3 > 2\sqrt{2} = f(1, 2) + f(2, 1)
  \end{math}.
\end{example}

\begin{example}
  \label{ex:subadd-ca-surjective}
  The function (\ref{eq:subadd-ca-logoutput}) of Example
  \ref{ex:subadd-subshift-logoutput} is not, in general,
  subadditive. For example, for $d=2$ and $X=A^{\integers^2}$ every
  pattern is allowed, so
  \begin{math}
    f(n_1, n_2) = n_1 n_2
  \end{math}:
  but if $n_1$, $n_2$, $m_1$, and $m_2$ are all positive, then
  \begin{math}
    (n_1 + m_1) (n_2 + m_2) > n_1 n_2 + m_1 m_2
  \end{math}.
\end{example}

Although componentwise subadditivity is very different from
subadditivity in the product semigroup, Fekete's lemma can tell us
something important for the case of positive integer or real
variables.

\begin{lemma}
  \label{lem:positive-variables}
  Let
  \begin{math}
    S = \prod_{i=1}^d S_i
  \end{math}
  with each $S_i$ being either $\positives$ or $\positiveint$, and let
  $f:S^d\to\reals\cup\{-\infty\}$ be componentwise subadditive. Having
  fixed $k\in\slice{1}{d-1}$, let
  \begin{math}
    i, j_1, \ldots, j_k \in \slice{1}{d}
  \end{math}
  be pairwise different. However fixed the values of the remaining
  variables, the function
  \begin{math}
    h : S_i \to \reals \cup \{ -\infty \}
  \end{math}
  defined by the multiple limit:
  \begin{equation}
    \label{eq:positive-variables}
    h(x_i) =
    \lim_{x_{j_1} \to +\infty} \ldots \lim_{x_{j_k} \to +\infty}
    \dfrac{f(x_1, \ldots, x_d)}{\prod_{i=1}^k x_{j_i}}
  \end{equation}
  is subadditive.
\end{lemma}

\begin{proof}
  It is sufficient to prove the thesis for $k=1$; the general case
  follows by repeated application. To simplify notation, let
  $j=j_1$. Fix the values of $x_s$ for
  \begin{math}
    s \in \slice{1}{d} \setminus \{ i, j \}
  \end{math}.
  By hypothesis, for every $x_i\in{S_i}$ the function
  \begin{math}
    \lambdaxt{x_j}{f(x_1, \ldots, x_d)}
  \end{math}
  is subadditive, so by Fekete's lemma
  \begin{math}
    h(x_i) = \lim_{x_j \to \infty} \frac{f(x_1, \ldots, x_d)}{x_j}
  \end{math}
  exists. But for every $x_i,x_i',x_j>0$ it is:
  \begin{displaymath}
    \frac{f(\ldots, x_i + x_i', \ldots )}{x_j}
    \leq
    \frac{f(\ldots, x_i, \ldots)}{x_j}
    + \frac{f(\ldots, x_1', \ldots)}{x_j}
    \,,
  \end{displaymath}
  so it must be
  \begin{math}
    h(x_1 + x_1') \leq h(x_1) + h(x_1')
  \end{math}
  too. Note that it is crucial for the proof that $x_j>0$.
\end{proof}

The following observation is crucial for the next sections.

\begin{proposition}
  \label{prop:subadditive-componentwise-orthant}
  Let
  \begin{math}
    w = w_1 \ldots w_d
  \end{math}
  be a binary word of length $d$ and let
  \begin{math}
    f : \reals_w \to \reals
  \end{math}.
  For every $i\in\slice{1}{d}$ let
  \begin{math}
    x_{w,i} = (-1)^{w_i} x_i \in \positives
  \end{math},
  and let
  \begin{math}
    f_w : \positives^d \to \reals
  \end{math}
  be defined by
  \begin{math}
    f_w(x_{w,1}, \ldots, x_{w,d}) = f(x_1, \ldots, x_d)
  \end{math}.
  The following are equivalent:
  \begin{enumerate}
  \item
    $f(x_1,\ldots,x_d)$ is componentwise subadditive in $\reals_w$.
  \item
    \begin{math}
      f_w(x_{w,1}, \ldots, x_{w,d})
    \end{math}
    is componentwise subadditive in $\positives^d$.
  \end{enumerate}
  The same holds if $\reals_w$ and $\positives^d$ are replaced with
  \begin{math}
    \integers_w = \reals_w \cap \integers^d
  \end{math}
  and $\positiveint^d$, respectively.
\end{proposition}

\section{%
  Componentwise subadditive functions of $d$ real variables are bounded on compacts%
}
\label{sec:boundedness}

In \cite{capobianco2008} we prove the following:

\begin{proposition}[%
    {Fekete's lemma in $\positiveint^d$; \cite[Theorem 1]{capobianco2008}}%
  ]
  \label{prop:fekete-positiveint-d}
  Let
  \begin{math}
    \Ucal = \left( \positiveint^d, \leq_\Pi \right)
  \end{math}
  and let
  \begin{math}
    f : \positiveint^d \to \reals
  \end{math}
  be componentwise subadditive. Then:
  \begin{equation}
    \label{eq:fekete-positiveint-d}
    \lim_{(x_1, \ldots, x_d) \to \Ucal} \frac{f(x_1, \ldots, x_d)}{x_1 \cdots x_d}
    =
    \inf_{x_1, \ldots, x_d \in \positiveint} \frac{f(x_1, \ldots, x_d)}{x_1 \cdots x_d}
    \,.
  \end{equation}
\end{proposition}

\begin{example}
  \label{ex:fekete-entropy}
  With the notation of Example \ref{ex:subadd-subshift-logoutput} and
  $\Ucal$ as in Proposition \ref{prop:fekete-positiveint-d}, the
  value:
  \begin{equation}
    \label{eq:fekete-entropy}
    h(X)
    =
    \lim_{(x_1, \ldots, x_d) \to \Ucal}
    \frac{\log_a \Acal_X(x_1, \ldots, x_d)}{x_1 \cdots x_d}
  \end{equation}
  is well defined, and is called the \emph{entropy} of the subshift
  $X$. For $d=1$ this coincides with \cite[Definition 4.1.1]{lm1995}.
\end{example}

We try to reuse the argument from \cite{capobianco2008} to prove
Proposition \ref{prop:fekete-2d-positives}. Fix $s,t>0$. Every $x>0$
large enough has a unique writing
\begin{math}
  x = qs + r
\end{math}
with $q$ positive integer and
\begin{math}
  r \in \left[ s, 2s \right)
\end{math},
and every $y>0$ large enough has a unique writing
\begin{math}
  y = mt + p
\end{math}
with $m$ positive integer and
\begin{math}
  p \in \left[ t, 2t \right)
\end{math}. By subadditivity,
\begin{eqnarray*}
  \frac{f(x, y)}{x \cdot y}
  & \leq &
  \frac{q}{x \cdot y} f(s, y)
  + \frac{1}{x \cdot y} f(r, y)
  \\ & \leq &
  \frac{q}{x} \cdot \frac{m}{y} \cdot f(s, t)
  \\ & &
  + \frac{q}{x} \cdot \frac{1}{y} \cdot f(s, p)
  + \frac{1}{x} \cdot \frac{m}{y} \cdot f(r, t)
  \\ & &
  + \frac{1}{x} \cdot \frac{1}{y} \cdot f(r, p)
\end{eqnarray*}
Consider the four summands on the right-hand side of the last
inequality. By construction,
\begin{math}
  \lim_{x \to +\infty} q / x = 1 / s
\end{math}
and
\begin{math}
  \lim_{y \to +\infty} m / y = 1 / t
\end{math}:
therefore, the first summand converges to
\begin{math}
  f(s, t) / st
\end{math}
for
\begin{math}
  (x, y) \to \mainorthant^2
\end{math}.

Now, by \cite[Theorem 6.4.1]{hille1948}, a subadditive function of one
positive real variable is bounded in every compact subset of
$\positives$. Then
\begin{math}
  \lambdaxt{p}{f(s, p)}
\end{math}
is bounded on
\begin{math}
  \left[ t, 2t \right]
\end{math}
and
\begin{math}
  \lambdaxt{r}{f(r, t)}
\end{math}
is bounded on
\begin{math}
  \left[ s, 2s \right]
\end{math}:
consequently, the second and third summand vanish for
\begin{math}
  (x, y) \to \mainorthant^2
\end{math}.

But the fourth summand presents a problem. What we know, is that
\begin{math}
  \lambdaxt{x}{f(x, y)}
\end{math}
is bounded in $[s,2s]$ for every $y\in[t,2t]$, and
\begin{math}
  \lambdaxt{y}{f(s, y)}
\end{math}
is bounded in $[t,2t]$ for every $x\in[s,2s]$. This is, in general,
\emph{strictly less} than $f$ being bounded in
\begin{math}
  [s, 2s] \times [t, 2t]
\end{math}:
which is what we actually need to show that the fourth summand
vanishes when $x$ and $y$ both grow arbitrarily large!

\begin{example}[{suggested by Arthur Rubin}]
  \label{ex:bound-unbound}
  Let
  \begin{math}
    h : \positives \to \reals
  \end{math}
  be such that $h(t)$ is the denominator of the representation of $t$
  as an irreducible fraction if $t$ is rational, and 0 if $t$ is
  irrational. Then
  \begin{math}
    f : \positives^2 \to \reals
  \end{math}
  defined by
  \begin{math}
    f(x, y) = \min(h(x), h(y))
  \end{math}
  satisfies the following conditions:
  \begin{enumerate}
  \item
    for every $x\in[1,2]$, the function $\lambdaxt{y}{f(x,y)}$ is
    bounded in $[1,2]$;
  \item
    for every $y\in[1,2]$, the function $\lambdaxt{x}{f(x,y)}$ is
    bounded in $[1,2]$.
  \end{enumerate}
  However, $f$ is not bounded in
  \begin{math}
    [1, 2] \times [1, 2]
  \end{math},
  because
  \begin{math}
    f(1 + 1/n, 1 + 1/n)
    = n
  \end{math}
  for every $n\in\positiveint$. On the other hand, $h(4)=1$ and
  \begin{math}
    h(\pi) = h(4-\pi) = 0
  \end{math},
  so $f$ is neither subadditive nor componentwise subadditive in
  $\positives^2$.
\end{example}

We could overcome this issue if a result of boundedness such as the
one in \cite[Theorem 6.4.1]{hille1948} held for componentwise
subadditive functions. Luckily, it is so, and we can follow the same
idea of Hille's proof. Given
\begin{math}
  f : \positives^d \to \reals
\end{math}
and
\begin{math}
  t_1, \ldots, t_d \in \positives
\end{math},
let:
\begin{equation}
  \label{eq:levelset-d}
  V_{t_1, \ldots, t_d, k}
  = \{
  (x_1, \ldots, x_d) \in \positives^d \mid
  0 < x_i < t_i \forall i \in \slice{1}{d},
  f(x_1, \ldots, x_d) \geq k
  \} \,.
\end{equation}
Under our hypothesis that $f$ is measurable, so is
(\ref{eq:levelset-d}).

The next statement is the cornerstone of our argument. For Lemma
\ref{lem:subadd-levelset-d} and Theorem \ref{thm:subadd-bounded-1q},
the symbol $\mu$ and the word ``measure'' denote the $d$-dimensional
Lebesgue measure.

\begin{lemma}
  \label{lem:subadd-levelset-d}
  Let
  \begin{math}
    f : \positives^d \to \reals
  \end{math}
  be componentwise subadditive. Then for every
  \begin{math}
    t_1, \ldots, t_d \in \positives
  \end{math},
  \begin{equation}
    \label{eq:subadd-levelset-d}
    \mu \left( V_{t_1, \ldots, t_d,  \frac{f(t_1, \ldots, t_d)}{2^d}} \right)
    \geq
    \frac{t_1 \cdots t_d}{2^d} \,.
  \end{equation}
\end{lemma}

\begin{proof}
  Call $V$ the set on the left-hand side of
  (\ref{eq:subadd-levelset-d}). For every $i\in\slice{1}{d}$, given
  \begin{math}
    x_i \in (0, t_i)
  \end{math},
  let
  \begin{math}
    y_i^{(0)} = x_i
  \end{math}
  and
  \begin{math}
    y_i^{(1)} = t_i - x_i
  \end{math}.
  Then for every $w\in\deux^d$ the transformation
  \begin{displaymath}
    \lambdaxt{%
      (x_1, \ldots, x_d)
    }{
      \left( y_1^{(w_1)}, \ldots, y_d^{(w_d)} \right)
    }
  \end{displaymath}
  is a measure-preserving continuous involution, hence the set:
  \begin{displaymath}
    V_w = \left\{ \left( y_1^{(w_1)}, \ldots, y_d^{(w_d)} \right) \mid
    (x_1, \ldots, x_d) \in V \right\}
  \end{displaymath}
  is measurable and satisfies
  \begin{math}
    \mu(V_w) = \mu(V)
  \end{math}.
  Note that $V=V_{0^d}$.

  By repeatedly applying subadditivity, once in each variable, we
  arrive at:
  \begin{equation}
    \label{eq:levelset-subadd}
    f(t_1, \ldots, t_d)
    \leq
    \sum_{w \in \deux^d} f \left( y_1^{(w_1)}, \ldots, y_d^{(w_d)} \right)
    \,.
  \end{equation}
  For example, for $d=2$ we have:
  \begin{eqnarray*}
    f(t_1, t_2)
    & \leq &
    f(x_1, t_2) + f(t_1 - x_1, t_2)
    \\ & \leq &
    f(x_1, x_2) + f(x_1, t_2 - x_2)
    \\ & &
    + f(t_1 - x_1, x_2) + f(t_1 - x_1, t_2 - x_2)
    \,.
  \end{eqnarray*}
  For (\ref{eq:levelset-subadd}) to hold, at least one of the $2^d$
  summands on the right-hand side must be no smaller than
  \begin{math}
    \dfrac{f(t_1, \ldots, t_d)}{2^d}
  \end{math}.
  Then
  \begin{math}
    \bigcup_{w \in \deux^d} V_w = \prod_{i=1}^d (0, t_i)
  \end{math},
  so:
  \begin{displaymath}
    t_1 \cdots t_d
    \leq
    \sum_{w \in \deux^d} \mu(V_w)
    =
    2^d \, \mu(V) \,.
  \end{displaymath}
\end{proof}

From Lemma \ref{lem:subadd-levelset-d} follows:

\begin{theorem}
  \label{thm:subadd-bounded-1q}
  Let
  \begin{math}
    f : \positives^d \to \reals
  \end{math}
  be componentwise subadditive. Then $f$ is bounded in every compact
  subset of $\positives^d$.
\end{theorem}

\begin{proof}
  It is sufficient to prove the thesis for every compact hypercube of
  the form $H=[a,b]^d$ with $0<a<b$. We proceed by contradiction,
  following the argument from \cite[Theorem 6.4.1]{hille1948}.

  First, suppose that $f$ is unbounded from above in $H$. Then for
  every $n\geq1$ and $i\in\slice{1}{d}$ there exists
  \begin{math}
    x_{i, n} \in [a, b]
  \end{math}
  such that
  \begin{math}
    f(x_{1, n}, \ldots, x_{d, n}) \geq 2^d n
  \end{math}.
  Let
  \begin{math}
    W_{t_1, \ldots, t_d}
  \end{math}
  be the set in (\ref{eq:subadd-levelset-d}). By construction, for
  every $n\geq1$ we have
  \begin{displaymath}
    W_{x_{1,n}, \ldots, x_{d,n}}
    \subseteq
    V_{b, \ldots, b, n}
    \,,
  \end{displaymath}
  and by Lemma \ref{lem:subadd-levelset-d},
  \begin{displaymath}
    \mu \left( W_{x_{1,n}, \ldots, x_{d,n}} \right)
    \geq
    \frac{x_{1,n} \cdots x_{d,n}}{2^d}
    \geq
    \left( \frac{a}{2} \right)^d \,.
  \end{displaymath}
  Now, the sets
  \begin{math}
    V_{b, \ldots, b, n}
  \end{math}
  are measurable and form a nonincreasing sequence, so
  \begin{math}
    V = \bigcap_{n \geq 0} V_{b, \ldots, b, n}
  \end{math}
  is measurable and
  \begin{math}
    \mu(V) \geq\left( a/2 \right)^d
  \end{math}:
  in particular, $V$ cannot be empty. But for
  \begin{math}
    (x_1, \ldots, x_d) \in V
  \end{math}
  it must be
  \begin{math}
    f(x_1, \ldots, x_d) \geq n
  \end{math}
  for every $n\geq1$: which is impossible.

  Next, suppose that $f$ is unbounded from below in $H$. Then for
  every $n\geq1$ and $i\in\slice{1}{d}$ there exists
  \begin{math}
    x_{i, n} \in [a, b]
  \end{math}
  such that
  \begin{math}
    f(x_{1, n}, \ldots, x_{d, n}) \leq -n
  \end{math}:
  we may assume that
  \begin{math}
    \lim_{n \to \infty} x_{i, n} = x_i \in [a, b]
  \end{math}
  exists for every $i\in\slice{1}{d}$. Let
  \begin{math}
    s = \min(a, 1)
  \end{math},
  \begin{math}
    t = b + 4
  \end{math},
  and
  \begin{math}
    J = [s,t]^d
  \end{math}:
  then every point
  \begin{math}
    (z_1, \ldots, z_d)
  \end{math}
  where each $z_i$ belongs to either $[a,b]$ or $[1,4]$ belongs to
  $J$. Let now $y_i\in[1,4]$ for every $i\in\slice{1}{d}$ and
  \begin{displaymath}
    M = \sup \{ f(z_1, \ldots, z_d) \mid (z_1, \ldots, z_d) \in J \} \,,
  \end{displaymath}
  which is a real number because of the previous point. By applying
  subadditivity in each variable, for such
  \begin{math}
    y_1, \ldots, y_d
  \end{math}
  and $n$ we obtain
  \begin{displaymath}
    \label{eq:subadd-bounded-1q}
    f(y_1 + x_{1, n}, \ldots, y_d + x_{d, n})
    \leq
    (2^d-1)M - n
    \,,
  \end{displaymath}
  because $-n$ is an upper bound for
  \begin{math}
    f(x_{1, n}, \ldots, x_{d, n})
  \end{math}
  and $M$ is an upper bound for the other $2^d-1$ summands, For
  example, for $d=2$ we have:
  \begin{eqnarray*}
    f(y_1 + x_{1, n}, y_2 + x_{2, n})
    & \leq &
    f(y_1, y_2) + f(y_1, x_{2, n})
    \\ & &
    + f(x_{1, n}, y_2) + f(x_{1, n}, x_{2, n})
    \\ & \leq & 3M - n \,.
  \end{eqnarray*}
  But for every $n$ such that
  \begin{math}
    |x_{i,n} - x_i| \leq 1
  \end{math}
  it is
  \begin{math}
    [x_i + 2, x_i + 3] \subseteq [x_{i,n} + 1, x_{i,n} + 4]
  \end{math}:
  calling
  \begin{displaymath}
    K = \prod_{i=1}^d [x_i + 2, x_i + 3] \subseteq J \,,
  \end{displaymath}
  for every $n$ large enough every element of $K$ can be written in
  the form
  \begin{math}
    (y_1 + x_{1,n}, \ldots, y_d + x_{d,n})
  \end{math}
  for suitable
  \begin{math}
    y_1, \ldots, y_d \in [1, 4]
  \end{math}.
  For every
  \begin{math}
    (z_1, \ldots, z_d) \in K
  \end{math}
  it must then be
  \begin{math}
    f(z_1, \ldots, z_d) \leq (2^d-1)M - n
  \end{math}
  for every $n$ large enough: which is impossible.
\end{proof}

If $f:\reals_w\to\reals$ is bounded on the compact subsets of
$\reals_w$, then $f_w$ as defined in Proposition
\ref{prop:subadditive-componentwise-orthant} is bounded on the compact
subsets of $\positives$; and vice versa.  From Theorem
\ref{thm:subadd-bounded-1q} and Proposition
\ref{prop:subadditive-componentwise-orthant} follows:

\begin{corollary}
  \label{cor:subadd-bounded-Rw}
  Let $w\in\deux^d$ and let
  \begin{math}
    f : \reals_w \to \reals
  \end{math}
  be componentwise subadditive. Then $f$ is bounded in every compact
  subset of $\reals_w$.
\end{corollary}

In turn, Corollary \ref{cor:subadd-bounded-Rw} allows us to prove:

\begin{theorem}
  \label{thm:subadd-bounded}
  Let
  \begin{math}
    f : \reals^d \to \reals
  \end{math}
  be componentwise subadditive. Then $f$ is bounded in every compact
  subset of $\reals^d$.
\end{theorem}

\begin{proof}
  It is sufficient to show finitely many open sets
  \begin{math}
    V_1, \ldots, V_n
  \end{math}
  such that $f$ is bounded on the compacts of each $V_i$ and:
  \begin{displaymath}
    \reals^d
    =
    \left( \bigcup_{w \in \deux^d} \reals_w \right)
    \cup
    \left( \bigcup_{i=1}^n V_i \right) \,.
  \end{displaymath}
  We give the argument for $d=3$: the ideas for arbitrary $d\geq1$ are
  similar. Let
  \begin{math}
    I = [-1/2, 1/2]
  \end{math}
  and
  \begin{math}
    U = [-3/2, -1/2] \cup [1/2, 3/2]
  \end{math}.

  We start by proving that $f$ is bounded in every compact subset of
  the open set
  \begin{displaymath}
    Z_{00}
    =
    \{ (x, y, z) \in \reals^3 \mid x>0, y>0 \}
    =
    \reals_{000} \cup \reals_{001} \cup D_{00} \,,
  \end{displaymath}
  where
  \begin{math}
    D_{00} = \{ (x, y, z) \in \reals^3 \mid x>0, y>0, z=0 \}
  \end{math}
  is the first quadrant of the $XY$ plane. To do this, we only need to
  show that $f$ is bounded in every set of the form
  \begin{math}
    H = [a, b] \times [a, b] \times I
  \end{math}.
  Let
  \begin{math}
    V = [a, b] \times [a, b] \times U
  \end{math}:
  if $(x,y,z)\in{H}$, then $(x,y,z-1)$ and $(x,y,z+1)$ are both in
  $V$. Let $T$ and $t$ be an upper bound and a lower bound for $f$ in
  $V$, respectively: then for every $(x,y,z)\in{H}$,
  \begin{displaymath}
    f(x, y, z)
    \leq
    f(x, y, z-1) + f(x, y, 1)
    \leq
    2T
  \end{displaymath}
  and
  \begin{displaymath}
    f(x, y, z)
    \geq
    f(x, y, z+1) - f(x, y, 1)
    \geq
    t - T \,.
  \end{displaymath}
  By similar arguments, $f$ is bounded in every compact subset of
  every subset of $\reals^3$ which is the union of two adjacent
  orthants and the corresponding ``quadrant''. As for each open orthant
  there are three which border it by one ``quadrant'', there are
  \begin{math}
    \dfrac{8 \cdot 3}{2} = 12
  \end{math}
  such subsets.

  We now show that $f$ is bounded in every compact subset of the open
  ``upper demispace''
  \begin{math}
    Z_0 = \{ (x, y, z) \in \reals^3 \mid z > 0 \}
  \end{math}.
  To do so, it will suffice to show that $f$ is bounded in every set
  of the form
  \begin{math}
    L = I \times I \times [a, b]
  \end{math}
  with $0<a<b$. Let
  \begin{math}
    W = U \times U \times [a, b]
  \end{math}
  and let $S$ and $s$ be an upper bound for $f$ in $W$, respectively:
  then for every $(x,y,z)\in{L}$,
  \begin{eqnarray*}
    f(x, y, z)
    & \leq &
    f(x-1, y, z) + f(1, y, z)
    \\ & \leq &
    f(x-1, y-1, z) + f(x-1, 1, z) + f(1, y-1, z) + f(1, 1, z)
    \\ & \leq &
    4S
  \end{eqnarray*}
  and
  \begin{eqnarray*}
    f(x, y, z)
    & \geq &
    f(x+1, y, z) - f(1, y, z)
    \\ & \geq &
    f(x+1, y+1, z) - f(x+1, 1, z) - f(1, y, z)
    \\ & \geq &
    s - 2S \,.
  \end{eqnarray*}
  Similarly, $f$ is bounded in each of the other five open
  ``demispaces''.

  To conclude the proof, we only need to show that $f$ is bounded in
  \begin{math}
    K = I \times I \times I
  \end{math}.
  Let
  \begin{math}
    E = U \times U \times U
  \end{math}
  and let $M$ and $m$ be an upper bound and a lower bound for $f$ in
  $E$, respectively: then for every $(x,y,z)\in{K}$,
  \begin{eqnarray*}
    f(x, y, z)
    & \leq &
    f(x-1, y, z) + f(1, y, z)
    \\ & \leq &
    f(x-1, y-1, z) + f(x-1, 1, z) + f(1, y-1, z) + f(1, 1, z)
    \\ & \leq &
    f(x-1, y-1, z-1) + f(x-1, z-1, 1)
    \\ & &
    + f(x-1, 1, z-1) + f(x-1, 1, 1)
    \\ & &
    + f(1, y-1, z-1) + f(1, y-1, 1)
    \\ & &
    + f(1, 1, z-1) + f(1, 1, 1)
    \\ & \leq & 8M
  \end{eqnarray*}
  and
  \begin{eqnarray*}
    f(x, y, z)
    & \geq &
    f(x+1, y, z) - f(1, y, z)
    \\ & \geq &
    f(x+1, y+1, z) - f(x+1, 1, z) - f(1, y, z)
    \\ & \geq &
    f(x+1, y+1, z+1) - f(x+1, y+1, 1)
    \\ & &
    - f(x+1, 1, z) - f(1, y, z)
    \\ & \geq &
    m - 3M \,.
  \end{eqnarray*}
\end{proof}

Note that the argument of Lemma \ref{lem:subadd-levelset-d} also works
if $f$ is subadditive, rather than componentwise subadditive. In this
case, however, the denominator in (\ref{eq:subadd-levelset-d}) and in
the thesis is 2 rather than $2^d$. A more complex variant of it can
then be stated, where $f$ is a function of $k$ variables $x_i$, each
taking values in an orthant of $\positives^{d_i}$: and the denominator
would then be $2^k$. From this, a generalization of Theorem
\ref{thm:subadd-bounded-1q} to the case of componentwise functions of
$k$ variables, the $i$th of which takes values in $\positives^{d_i}$,
can be derived.

\section{%
  Fekete's lemma for componentwise subadditive functions of $d$ real variables%
}
\label{sec:fekete-d}
  
We can now state and prove the main result of this paper.

\begin{theorem}[Fekete's lemma in $\positives^d$]
  \label{thm:fekete-positives-d}
  Let $d\geq1$ and let
  \begin{math}
    f : \positives^d \to \reals
  \end{math}
  be componentwise subadditive. Then:
  \begin{equation}
    \label{eq:fekete-positives-d}
    \lim_{(x_1, \ldots, x_d) \to \mainorthant^{d}}
    \frac{f(x_1, \ldots, x_d)}{x_1 \cdots x_d}
    =
    \inf_{x_1, \ldots, x_d \in \positives}
    \frac{f(x_1, \ldots, x_d)}{x_1 \cdots x_d}
    \,,
  \end{equation}
  which can be $-\infty$. In addition, for every permutation $\sigma$
  of $\slice{1}{d}$,
  \begin{equation}
    \label{eq:fekete-positives-d-multiplelimit}
    \lim_{x_{\sigma(1)} \to +\infty} \ldots
    \lim_{x_{\sigma(d)} \to +\infty}
    \frac{f(x_1, \ldots, x_d)}{x_1 \cdots x_d}
    =
    \lim_{(x_1, \ldots, x_d) \to \mainorthant^{d}}
    \frac{f(x_1, \ldots, x_d)}{x_1 \cdots x_d} \,,
  \end{equation}
  regardless of any of the limits being finite or (negatively)
  infinite.
\end{theorem}

The proof of (\ref{eq:fekete-positives-d}) is similar to that of
\cite[Theorem 1]{capobianco2008}, with an important change; for
convenience of the reader, we report it entirely. The proof of
(\ref{eq:fekete-positives-d-multiplelimit}) relies on
(\ref{eq:fekete-positives-d}), the original Fekete's lemma, and the
following Lemma \ref{lem:multiple-inf}; the argument remains valid for
the case of positive integer variables.

\begin{lemma}
  \label{lem:multiple-inf}
  Let $u$ be a real-valued function depending on $d\geq1$ variables
  $x_i$, no matter of what type. Then for every permutation $\sigma$
  of $\slice{1}{d}$,
  \begin{equation}
    \label{eq:multiple-inf}
    \inf_{x_{\sigma(1)}} \ldots \inf_{x_{\sigma(d)}} u(x_1, \ldots, x_d)
    =
    \inf_{x_1, \ldots, x_d} u(x_1, \ldots, x_d) \,.
  \end{equation}
\end{lemma}

\begin{proof}[Sketch of proof.]
  Let $L$ and $R$ be the left- and right-hand side of
  (\ref{eq:multiple-inf}), respectively. It is easy to see that
  $L\geq{R}$. Let now $\delta>0$: there exist
  $z_1,\ldots,z_d\in\positives$ such that
  \begin{math}
    u(z_1, \ldots, z_d) < R + \delta
  \end{math},
  so a fortiori $L<R+\delta$ too; as $\delta>0$ is arbitrary,
  $L\leq{R}$.
\end{proof}

\begin{proof}[Proof of Theorem \ref{thm:fekete-positives-d}]
  For every
  \begin{math}
    i \in \{ 1, \ldots, d \}
  \end{math}
  and
  \begin{math}
    x_1, \ldots, x_d \in \positives,
  \end{math}
  if $x_i=qt+r$ with $q\in\positiveint$ and $r,t\in\positives$, then:
  \begin{displaymath}
    f(x_1, \ldots, x_i, \ldots, x_d)
    \leq
    q \cdot f(x_1, \ldots, t, \ldots, x_d)
    + f(x_1, \ldots, r, \ldots, x_d)
    \;.
  \end{displaymath}
  Fix
  \begin{math}
    t_1, \ldots, t_d \in \positives
  \end{math}.
  For every $i\in\slice{1}{d}$ and $x_i\geq{2t_i}$ there exist unique
  $q_i\in\positiveint$ and
  \begin{math}
    r_i \in \left[ t_i, 2 t_i \right)
  \end{math}
  such that
  \begin{math}
    x_i = q_i t_i + r_i
  \end{math}.
  For every $i\in\slice{1}{d}$ let
  \begin{math}
    y_i^{(0)} = r_i
  \end{math}
  and
  \begin{math}
    y_i^{(1)} = t_i
  \end{math}:
  by repeatedly applying subadditivity, once per each variable, we
  find:
  \begin{equation}
    \label{eq:fekete-d-estimate}
    f(x_1, \ldots, x_d)
    \leq
    \sum_{w \in \deux^d}
    q_1^{w_1} \cdots q_d^{w_d}
    \cdot f \left( y_1^{(w_1)}, \ldots, y_d^{(w_d)} \right)
    \;.
  \end{equation}
  Now, on the right-hand side of (\ref{eq:fekete-d-estimate}), each
  occurrence of $f$ has $k$ arguments chosen from the $t_i$'s and
  $d-k$ chosen from the $r_i$'s, is multiplied by the $q_i$'s
  corresponding to the $t_i$'s, and is bounded from above by the
  constant
  \begin{displaymath}
    M = \sup \{
    f(y_1, \ldots, y_d) \mid y_i \in [t_i,2t_i] \; \forall i \in \slice{1}{d}
    \} \,,
  \end{displaymath}
  which exists because of Theorem \ref{thm:subadd-bounded-1q}. Such
  boundedness is crucial for the proof, and was ensured for free in
  the case of positive integer variables from \cite{capobianco2008},
  but had to be proved for positive real variables.

  By dividing both sides of (\ref{eq:fekete-d-estimate}) by
  $x_1\cdots{x_d}$ we get:
  \begin{equation}
    \label{eq:fekete-d-estimate2}
    \frac{f(x_1, \ldots, x_d)}{x_1 \cdots x_d}
    \leq
    \frac{q_1 \cdots q_d}{x_1 \cdots x_d} f(t_1,\ldots,t_d)
    + M \cdot \sum_{w \in \deux^d \setminus \{1^d\}}
    \frac{q_1^{w_1} \cdots q_d^{w_d}}{x_1\cdots x_d}
    \;.
  \end{equation}
  where $1^d$ is the binary word of length $d$ where all the bits are 1.

  By construction,
  \begin{math}
    \lim_{x_i \to \infty} q_i / x_i = 1/t_i
  \end{math}.
  Given $\varepsilon>0$, choose
  \begin{math}
    x_1^{(\varepsilon)}, \ldots, x_d^{(\varepsilon)} \in \positives
  \end{math}
  such that, if
  \begin{math}
    x_i \geq x_i^{(\varepsilon)}
  \end{math}
  for each $i\in\slice{1}{d}$, then the following hold:
  \begin{enumerate}
  \item
    \begin{math}
      \dfrac{f(x_1, \ldots, x_d)}{x_1 \cdots x_d}
      <
      \dfrac{f(t_1, \ldots, t_d)}{t_1 \cdots t_d}
      + \dfrac{\varepsilon}{2^d}
    \end{math};
  \item
    \begin{math}
      \dfrac{q_1^{w_1} \cdots q_d^{w_d}}{x_1\cdots x_d}
      < \dfrac{\varepsilon}{M \cdot 2^d}
    \end{math}
    for every
    \begin{math}
      w \in \deux^d \setminus \{ 1^d \}
    \end{math}.
  \end{enumerate}
  This is possible because if
  \begin{math}
    w \neq 1^d
  \end{math},
  then at least one of the $q_i^{w_i}$ equals 1. For such
  \begin{math}
    x_1, \ldots, x_d
  \end{math}
  it is:
  \begin{displaymath}
    \frac{f(x_1, \ldots, x_d)}{x_1 \cdots x_d}
    <
    \frac{f(t_1, \ldots, t_d)}{t_1 \cdots t_d} + \varepsilon
    \;.
  \end{displaymath}
  As $\varepsilon>0$ is arbitrary, it must be:
  \begin{displaymath}
    \limsup_{(x_1, \ldots,x_d) \to \mainorthant^d}
    \frac{f(x_1, \ldots, x_d)}{x_1 \cdots x_d}
    \leq
    \frac{f(t_1, \ldots, t_d)}{t_1 \cdots t_d}
    \,.
  \end{displaymath}
  But the $t_i$'s are also arbitrary, hence:
  \begin{eqnarray*}
    \limsup_{(x_1, \ldots, x_d) \to \mainorthant^d}
    \frac{f(x_1 ,\ldots, x_d)}{x_1 \cdots x_d}
    & \leq &
    \inf_{t_1, \ldots, t_d \in \positives}
    \frac{f(t_1, \ldots, t_d)}{t_1 \cdots t_d}
    \\ & \leq &
    \liminf_{(x_1, \ldots, x_d) \to \mainorthant^d}
    \frac{f(x_1, \ldots, x_d)}{x_1 \cdots x_d}
    \;,
  \end{eqnarray*}
  which yields (\ref{eq:fekete-positives-d}).

  Now, by Lemma \ref{lem:positive-variables}, for every choice of
  \begin{math}
    i, j_1, \ldots, j_k \in \slice{1}{d}
  \end{math}
  all different, and however fixed the remaining variables, the
  function (\ref{eq:positive-variables}) is subadditive. Then
  (\ref{eq:fekete-positives-d-multiplelimit}) follows from Lemma
  \ref{lem:multiple-inf} by repeated application of the original
  Fekete's lemma:
  \begin{eqnarray*}
    \lim_{x_{\sigma(1)} \to +\infty} \ldots
    \lim_{x_{\sigma(d)} \to +\infty}
    \frac{f(x_1, \ldots, x_d)}{x_1 \cdots x_d}
    & = &
    \inf_{x_{\sigma(1)} > 0}
    \lim_{x_{\sigma(2)} \to +\infty} \ldots
    \lim_{x_\sigma(d) \to \infty}\frac{f(x_1, \ldots, x_d)}{x_1 \cdots x_d}
    \\ & = & \ldots
    \\ & = &
    \inf_{x_{\sigma(1)} > 0} \ldots
    \inf_{x_{\sigma(d)} > 0}
    \frac{f(x_1, \ldots, x_d)}{x_1 \cdots x_d}
    \\ & = &
    \inf_{x_1, \ldots, x_d > 0}
    \frac{f(x_1, \ldots, x_d)}{x_1 \cdots x_d}
    \\ & = &
    \lim_{(x_1, \ldots, x_d) \to \mainorthant^{d}}
    \frac{f(x_1, \ldots, x_d)}{x_1 \cdots x_d}
    \,.
  \end{eqnarray*}
\end{proof}

From Theorem \ref{thm:fekete-positives-d} and Proposition
\ref{prop:subadditive-componentwise-orthant} follows:

\begin{theorem}
  \label{thm:fekete-Rd}
  Let $d\geq1$, let $w,w'\in\deux^d$ and let
  \begin{math}
    f : \reals_w \to \reals
  \end{math}
  be componentwise subadditive.
  \begin{enumerate}
  \item
    \label{it:fekete-Rd-even}
    If $w$ contains evenly many 1s, then:
    \begin{equation}
      \lim_{(x_1, \ldots, x_d) \to \orthant_w}
      \frac{f(x_1, \ldots, x_d)}{x_1 \cdots x_d}
      =
      \inf_{(x_1, \ldots, x_d) \in \reals_w}
      \frac{f(x_1, \ldots, x_d)}{x_1 \cdots x_d}
    \end{equation}
    is not $+\infty$, but can be $-\infty$.
  \item
    \label{it:fekete-Rd-odd}
    If $w$ contains oddly many 1s, then:
    \begin{equation}
      \lim_{(x_1, \ldots, x_d) \to \orthant_w}
      \frac{f(x_1, \ldots, x_d)}{x_1 \cdots x_d}
      =
      \sup_{(x_1, \ldots, x_d) \in \reals_w}
      \frac{f(x_1, \ldots, x_d)}{x_1 \cdots x_d}
    \end{equation}
    is not $-\infty$, but can be $+\infty$.
  \item
    \label{it:fekete-Rd-all}
    Suppose now $w$ contains evenly many 1s, $w'$ differs from $w$ in
    exactly one coordinate, and $f$ is defined and componentwise
    subadditive in
    \begin{math}
      \reals_w \cup \reals_{w'} \cup U_{w,w'}
    \end{math},
    where:
    \begin{displaymath}
      U_{w,w'} = \{ x \in \reals^d
      \mid x_i=0,
      (x_1, \ldots, x_{i-1}, 1, x_{i+1}, \ldots, x_d)
      \in \reals_w \cup \reals_{w'}
      \}
    \end{displaymath}
    is the boundary between $\reals_w$ and $\reals_{w'}$. Then:
    \begin{equation}
      \lim_{(x_1, \ldots, x_d) \to \orthant_{w'}}
      \frac{f(x_1, \ldots, x_d)}{x_1 \cdots x_d}
      \leq
      \lim_{(x_1, \ldots, x_d) \to \orthant_w}
      \frac{f(x_1, \ldots, x_d)}{x_1 \cdots x_d} \,;
    \end{equation}
    consequently, both limits are finite.
  \end{enumerate}
\end{theorem}

For $d=1$ we recover \cite[Theorem 6.6.1]{hille1948}. To prove Theorem
\ref{thm:fekete-Rd}, we make use of the following result, whose proof
we leave to the reader.

\begin{lemma}
  \label{lem:subadd-monoid}
  Let $S$ be a semigroup and
  \begin{math}
    f : S \to \reals
  \end{math}
  be a subadditive function. If $S$ is a monoid with identity $e$,
  then $f(e)\geq0$. If, in addition, $S$ is a group, then
  \begin{math}
    f(x) + f(x^{-1}) \geq 0
  \end{math}
  for every $x\in{S}$.
\end{lemma}

\begin{proof}[Proof of Theorem \ref{thm:fekete-Rd}]
  For
  \begin{math}
    (x_1, \ldots, x_d) \in \orthant_w
  \end{math}
  and $i\in\slice{1}{d}$ let $x_{w,i}$ and $f_w$ be defined as in
  Proposition \ref{prop:subadditive-componentwise-orthant}. If $w$
  contains evenly many 1s, then
  \begin{math}
    x_1 \cdots x_d = x_{w,1} \cdots x_{w,d}
  \end{math}
  and:
  \begin{eqnarray*}
    \lim_{x \to \orthant_w}
    \frac{f(x_1, \ldots, x_d)}{x_1 \cdots x_d}
    & = &
    \lim_{x_w \to \mainorthant^d}
    \frac{f_w(x_{w,1}, \ldots, x_{w,d})}{x_{w,1} \cdots x_{w,d}}
    \\ & = &
    \inf_{x_w \in \positives^d}
    \frac{f_w(x_{w,1}, \ldots, x_{w,d})}{x_{w,1} \cdots x_{w,d}}
    \\ & = &
    \inf_{x \in \reals_w}
    \frac{f(x_1, \ldots, x_d)}{x_1 \cdots x_d}
    \,.
  \end{eqnarray*}
  If $w$ contains oddly many 1s, then
  \begin{math}
    x_1 \cdots x_d = - x_{w,1} \cdots x_{w,d}
  \end{math}
  and:
  \begin{eqnarray*}
    \lim_{x \to \orthant_w} \frac{f(x_1, \ldots, x_d)}{x_1 \cdots x_d}
    & = &
    - \lim_{x_w \to \mainorthant^d}
    \frac{f_w(x_{w,1}, \ldots, x_{w,d})}{x_{w,1} \cdots x_{w,d}}
    \\ & = &
    - \inf_{x_w \in \positives^d}
    \frac{f_w(x_{w,1}, \ldots, x_{w,d})}{x_{w,1} \cdots x_{w,d}}
    \\ & = &
    \sup_{x \in \reals_w} \frac{f(x_1, \ldots, x_d)}{x_1 \cdots x_d}
    \,.
  \end{eqnarray*}
  Suppose now $w$ has evenly many 1s and $w'$ differs from $w$ only in
  component $i$, and $f$ is defined and componentwise subadditive in
  \begin{math}
    \reals_w \cup \reals_{w'} \cup U_{w,w'}
  \end{math}.
  Then for every
  \begin{math}
    x_1, \ldots, x_{i-1}, x_{i+1}, \ldots x_d
  \end{math}
  the function
  \begin{math}
    \lambdaxt{x_i}{f(x_1, \ldots, x_i, \ldots, x_d)}
  \end{math}
  is subadditive on $\reals$: by Lemma \ref{lem:subadd-monoid}, for
  every $x_i\in\reals$,
  \begin{displaymath}
    f(x_1, \ldots, x_i, \ldots, x_d)
    + f(x_1, \ldots, -x_i, \ldots, x_d)
    \geq 0 \,.
  \end{displaymath}
  Then:
  \begin{eqnarray*}
    & &
    \lim_{x \to \orthant_w} \frac{f(x_1, \ldots, x_d)}{x_1 \cdots x_d}
    - \lim_{x' \to \orthant_{w'}} \frac{f(x'_1, \ldots, x'_d)}{x'_1 \cdots x'_d}
    \\ & = &
    \lim_{x \to \orthant_w}
    \frac{f(x_1, \ldots, x_i, \ldots, x_d)}{x_1 \cdots x_d}
    + \lim_{x \to \orthant_{w}}
    \frac{f(x_1, \ldots, -x_i, \ldots, x_d)}{x_1 \cdots x_d}
    \\ & = &
    \lim_{x \to \orthant_w} \frac{
      f(x_1, \ldots, x_i, \ldots, x_d)
      + f(x_1, \ldots, -x_i, \ldots, x_d)
    }{
      x_1 \cdots x_d
    }
  \end{eqnarray*}
  is nonnegative. The last passage is valid because the two limits on
  the second line are either finite or $-\infty$.
\end{proof}

As every subnet of a convergent net converges to the same limit, we
get:

\begin{corollary}
  \label{cor:fekete-subsequence}
  Let
  \begin{math}
    f : \positives^d \to \reals
  \end{math}
  be componentwise subadditive and let $T$ be either $\positives$ or
  $\positiveint$. For every
  \begin{math}
    i \in \slice{1}{d}
  \end{math}
  let
  \begin{math}
    x_i(t) : T \to \positives
  \end{math}
  satisfy
  \begin{math}
    \lim_{t \to +\infty} x_i(t) = +\infty
  \end{math}.
  Then:
  \begin{equation}
    \label{eq:fekete-subsequence}
    \lim_{t \to +\infty}
    \frac{f(x_1(t), \ldots, x_d(t))}{x_1(t) \cdots x_d(t)}
    =
    \inf_{x_1, \ldots, x_d > 0} \frac{f(x_1, \ldots, x_d)}{x_1 \cdots x_d}
    \,,
  \end{equation}
  and also
  \begin{equation}
    \label{eq:fekete-subsequence-inf}
    \inf_{t \in T}
    \frac{f(x_1(t), \ldots, x_d(t))}{x_1(t) \cdots x_d(t)}
    =
    \inf_{x_1, \ldots, x_d > 0} \frac{f(x_1, \ldots, x_d)}{x_1 \cdots x_d}
    \,.
  \end{equation}
  In particular,
  \begin{equation}
    \label{eq:fekete-n-d}
    \lim_{n \to \infty} \frac{f(n, \ldots, n)}{n^d}
    =
    \inf_{n \geq 1} \frac{f(n, \ldots, n)}{n^d}
    =
    \inf_{x_1, \ldots, x_d > 0} \frac{f(x_1, \ldots, x_d)}{x_1 \cdots x_d}
    \,.
  \end{equation}
\end{corollary}

\begin{proof}[Sketch of proof]
  We only remark that (\ref{eq:fekete-subsequence-inf}) follows from
  (\ref{eq:fekete-positives-d}) and:
  \begin{eqnarray*}
    \inf_{x_1, \ldots, x_d > 0} \frac{f(x_1, \ldots, x_d)}{x_1 \cdots x_d}
    & \leq &
    \inf_{t \in T}
    \frac{f(x_1(t), \ldots, x_d(t))}{x_1(t) \cdots x_d(t)}
    \\ & \leq &
    \liminf_{t \to \infty}
    \frac{f(x_1(t), \ldots, x_d(t))}{x_1(t) \cdots x_d(t)} \,.
  \end{eqnarray*}
\end{proof}

Note that, in general, even if $f$ is componentwise subadditive on
$\positives^d$,
\begin{math}
  \lambdaxt{t}{f(t, \ldots, t)}
\end{math}
is not subadditive on $\positives$: a simple example is
\begin{math}
  f(x_1, x_2) = x_1 \cdot x_2
\end{math}.
This provides further evidence that Theorems
\ref{thm:fekete-positives-d} and \ref{thm:fekete-Rd} are not special
cases of \cite[Theorem 6.6.1]{hille1948}.

A real-valued function defined on a semigroup $(S,\cdot)$ is
\emph{superadditive} if it satisfies
\begin{math}
  f(x \cdot y) \geq f(x) + f(y)
\end{math}
for every $x,y\in{S}$. As $f$ is superadditive if and only if $-f$ is
subadditive, an analogue of Theorem \ref{thm:fekete-positives-d} holds
for componentwise superadditive functions, provided one swaps the
roles of $\inf$ and $\sup$ and those of $-\infty$ and $+\infty$. If
$f$ is superadditive in some variables and subadditive in other
variables, however, Theorem \ref{thm:fekete-positives-d} does not
hold.

\begin{example}
  \label{ex:superadd-subadd}
  The function
  \begin{math}
    f : \positives^2 \to \reals
  \end{math}
  defined by
  \begin{math}
    f(x_1, x_2) = x_1^2 \sqrt{x_2}
  \end{math}
  is superadditive in $x_1$ and subadditive in $x_2$, and
  \begin{math}
    f(x_1, x_2) / x_1 x_2 = x_1 / \sqrt{x_2}
  \end{math}.
  But
  \begin{math}
    \lim_{(x_1, x_2) \to \mainorthant^2} f(x_1, x_2) / x_1 x_2
  \end{math}
  does not exist, because for every $y,R>0$ there exist
  \begin{math}
    x_1, x_2 > R
  \end{math}
  such that
  \begin{math}
    x_1 / \sqrt{x_2} = y
  \end{math}.
  Also,
  \begin{math}
    \lim_{x_1 \to \infty} \lim_{x_2 \to \infty} \dfrac{f(x_1, x_2)}{x_1 \cdot x_2}
    = 0
  \end{math}
  but
  \begin{math}
    \lim_{x_2 \to \infty} \lim_{x_1 \to \infty} \dfrac{f(x_1, x_2)}{x_1 \cdot x_2}
    = +\infty
  \end{math}.
\end{example}

As a final remark for this section, the following statement appears in
the literature as an extension to arbitrary dimension of \cite[Theorem
  6.1.1]{hille1948}:
\begin{proposition}[{cf. \cite[Theorem 16.2.9]{kuczma2009}}]
  \label{prop:fekete-d-literature}
  Let
  \begin{math}
    f : \reals^d \to \reals
  \end{math}
  be subadditive \emph{in the variable $\mathbf{x}\in\reals^d$}. Then
  for every $\mathbf{x}\in\reals^d$ the following limit exists:
  \begin{displaymath}
    L_{\mathbf{x}} = \lim_{t \to +\infty} \frac{f(t\mathbf{x})}{t} \,.
  \end{displaymath}
\end{proposition}
This, however, is not so much an extension than a \emph{corollary}. If
\begin{math}
  f : \reals^d \to \reals
\end{math}
satisfies
\begin{math}
  f(\mathbf{x} + \mathbf{y}) \leq f(\mathbf{x}) + f(\mathbf{y})
\end{math}
for every
\begin{math}
  \mathbf{x}, \mathbf{y} \in \reals^d
\end{math},
then obviously
\begin{math}
  g_{\mathbf{x}}(t) = f(t\mathbf{x})
\end{math}
satisfies
\begin{math}
  g_{\mathbf{x}}(s + t) \leq g_{\mathbf{x}}(s) + g_{\mathbf{x}}(t)
\end{math}
for every $s,t>0$: and $L_{\mathbf{x}}$ is simply the limit of
$g_{\mathbf{x}}(t)/t$ according to \cite[Theorem 6.1.1]{hille1948}. On
the other hand, Theorem \ref{thm:fekete-Rd} is an extension.

\section{A comparison with the Ornstein-Weiss lemma}
\label{sec:owlemma}

A group $G$ is \emph{amenable} if there exist a directed set
\begin{math}
  \Ucal = (U, \preceq)
\end{math}
and a net
\begin{math}
  \{ F_x \}_{x \in U}
\end{math}
of finite nonempty subsets of $G$ such that:
\begin{equation}
  \label{eq:folner}
  \lim_{x \to \Ucal} \frac{| g F_x \setminus F_x |}{|F_x|} = 0
  \;\; \textrm{for every } g \in G \,.
\end{equation}
A net such as in (\ref{eq:folner}) is called a \emph{(left) F{\o}lner
  net} on the group $G$, from the Danish mathematician Erling
F{\o}lner who introduced them in \cite{folner1955}. Every abelian
group is amenable: for a proof, see \cite[Chapter 4]{csc10}.

\begin{proposition}[Ornstein-Weiss lemma; cf. \cite{ornstein-weiss}]
  \label{prop:ornstein-weiss}
  Let $G$ be an amenable group and let
  \begin{math}
    f : \pf(G) \to \reals
  \end{math}
  be a function which:
  \begin{enumerate}
  \item
    is subadditive with respect to set union, that is,
    \begin{math}
      f(A \cup B) \leq f(A) + f(B)
    \end{math}
    for every
    \begin{math}
      A, B \in \pf(G)
    \end{math};
    and
  \item
    satisfies
    \begin{math}
      f(A) = f(gA)
    \end{math}
    for every $A\in\pf(G)$ and $g\in{G}$.
  \end{enumerate}
  Then for every directed set
  \begin{math}
    \Ucal = (U, \preceq)
  \end{math}
  and every left F{\o}lner net
  \begin{math}
    \Fcal = \{ F_x \}_{x \in U}
  \end{math}
  on $G$,
  \begin{equation}
    \label{eq:ornstein-weiss}
    L = \lim_{x \to \Ucal} \frac{f(F_x)}{|F_x|}
  \end{equation}
  exists, and does not depend on the choice of $\Ucal$ and $\Fcal$.
\end{proposition}

The Ornstein-Weiss lemma says that, for ``well behaving'' functions on
amenable groups, a notion of \emph{asymptotic average} is well
defined. A detailed proof of Proposition \ref{prop:ornstein-weiss} is
given by F. Krieger in \cite{krieger2010}.

\begin{example}
  \label{ex:folner-entropy}
  Let $G$ be an amenable group and let $A$ be a finite set with
  $a\geq2$ elements. The \emph{shift} by $g\in{G}$ is the function
  \begin{math}
    \sigma_g : A^G \to A^G
  \end{math}
  defined by
  \begin{math}
    \sigma_g(c)(x) = c(g \cdot x)
  \end{math}
  for every $c\in{A^G}$ and $x\in{G}$. The notions of subshift and of
  allowed pattern with support $S\in\pf(G)$ are extended naturally
  from those of Example \ref{ex:subadd-subshift-logoutput}. Calling
  $\Acal_X(S)$ the number of allowed patterns for $X$ with support
  $S$, and convening that the unique \emph{empty pattern}
  $e:\emptyset\to{A}$ appears in every configuration, we have for
  every $S,T\in\pf(G)$:
  \begin{displaymath}
    \Acal_X(S)
    \leq \Acal_X(S \cup T)
    \leq \Acal_X(S) \cdot \Acal_X(T \setminus S)
    \leq \Acal_X(S) \cdot \Acal_X(T)
    \,.
  \end{displaymath}
  Indeed, every allowed pattern on $S$ (resp., $T\setminus{S}$) can
  be extended to at least one allowed pattern on $S\cup{T}$ (resp.,
  $T$) but joining an allowed pattern over $S$ and an allowed pattern
  over $T\setminus{S}$ does not necessarily yield an allowed pattern
  on $S\cup{T}$. Hence,
  \begin{math}
    f(S) = \log_a \Acal_X(S)
  \end{math}
  is subadditive on $\pf(G)$, and clearly satisfies $f(gS)=f(S)$ for
  every $g\in{G}$ and $S\in\pf(G)$. The \emph{entropy} of $X$ can then
  be defined as:
  \begin{equation}
    \label{eq:entropy-subshift}
    h(X) = \lim_{x \to \Ucal} \frac{\log_a \Acal_X(F_x)}{|F_x|}
  \end{equation}
  where
  \begin{math}
    \Ucal = (U, \preceq)
  \end{math}
  is an arbitrary directed set and
  \begin{math}
    \{ F_x \}_{x \in U}
  \end{math}
  is an arbitrary F{\o}lner net on $G$.
\end{example}

As the sets
\begin{math}
  E_{x_1, \ldots, x_d} = \prod_{i=1}^d \slice{1}{x_i}
\end{math}
with
\begin{math}
  x_1, \ldots, x_d \in \positiveint
\end{math}
constitute a F{\o}lner net on $\integers^d$, defining the entropy of a
$d$-dimensional subshift according to either Example
\ref{ex:fekete-entropy} or Example \ref{ex:folner-entropy} yields the
same result. Nevertheless, the Ornstein-Weiss lemma does not
generalize Fekete's lemma, nor it is possible to prove the latter from
the former, as the limit (\ref{eq:ornstein-weiss}) is only ensured to
exist, not to coincide with any specific value. In addition, even if
\begin{math}
  f : \integers \to \reals
\end{math}
is subadditive, the ``natural'' conversion
\begin{equation}
  \label{eq:subadd-ZtoPF}
  g(A) = \ifthenelse{A \neq \emptyset}{f(|A|)}{0}{,}
\end{equation}
where $|A|$ is the number of elements of $A$, is invariant by
translations, but needs not be subadditive on $\pf(\integers)$, the
main reason being that $|A\cup{B}|$ needs not equal
$|A|+|B|$. Moreover, while invariance by translations is essential in
the Ornstein-Weiss lemma, a translate of a subadditive function needs
not be subadditive.

\begin{example}
  \label{ex:ZtoPF-not}
  The function
  \begin{math}
    f(n) = n \!\! \mod 2
  \end{math}
  is easily seen to be subadditive on $\integers$. But the function
  $g$ defined from $f$ by (\ref{eq:subadd-ZtoPF}) is not subadditive
  on $\pf(\integers)$, because if
  \begin{math}
    U = \{ 1, 2 \}
  \end{math}
  and
  \begin{math}
    V = \{ 2, 3 \}
  \end{math},
  then $g(U\cup{V})=1$ and $g(U)=g(V)=0$. Note that $h(n)=f(n+1)$ is
  not subadditive, because $h(1)=0$ but $h(2)=1$.
\end{example}

\section{Conclusions}
\label{sec:conclusions}

We have discussed an extension of the notion of subadditivity in the
case of many independent variables. In this context, we have proved a
nontrivial extension of the classical Fekete's lemma to the case of
functions of $d\geq1$ real variables, which recovers the original
statement for $d=1$, and which is more general than other extensions
already present in the literature. While doing so, we have also proved
that these componentwise subadditive functions satisfy the important
property of being bounded on compact subsets, the case $d=1$ being
already known from the literature.

We believe that our results can be of interest for researchers in
economics, theory of dynamical systems, and mathematical analysis.

\bibliographystyle{plain}
\bibliography{biblio}

\end{document}